\documentclass[12pt]{article}
\usepackage{amsmath,amssymb,amsthm}
\usepackage{amsfonts, amsmath}
\usepackage[english]{babel}
\newtheorem{theorem}{Theorem}[section]
\newtheorem{lemma}{Lemma}[section]
\newtheorem{proposition}{Proposition}[section]

\theoremstyle{remark}
\newtheorem{remark}{Remark}[section]

\numberwithin{equation}{section}

\newcommand{\restr}{\lfloor}
\newcommand\eps\varepsilon

\newcommand{\weakto}{\rightharpoonup}

\newcommand{\R}{\mathbb R}

\newcommand{\s}{\cdot}
\newcommand{\beq}{\begin{equation}}
\newcommand{\eeq}{\end{equation}}
\newcommand{\barra}[1]{\overline{#1}}

\begin{document}
\title{\textbf{On a $(p,q)$-Laplacian problem with parametric concave term and asymmetric perturbation}}
\author{
\bf Salvatore A. Marano\thanks{Corresponding Author}, Sunra
J.N. Mosconi\\
\small{Dipartimento di Matematica e Informatica,
Universit\`a degli Studi di Catania,}\\
\small{Viale A. Doria 6, 95125 Catania, Italy}\\
\small{\it E-mail: marano@dmi.unict.it, mosconi@dmi.unict.it}\\
\mbox{}\\
\bf Nikolaos S. Papageorgiou\\
\small{Department of Mathematics, National Technical University of Athens,}\\
\small{Zografou Campus, 15780 Athens, Greece}\\
\small{\it E-mail: npapg@math.ntua.gr}
}
\date{}
\maketitle
\begin{abstract}
A Dirichlet problem driven by the $(p, q)$-Laplace operator and an asymmetric concave reaction with positive parameter is investigated. Four nontrivial smooth solutions (two positive, one negative, and the remaining nodal) are obtained once the parameter turns out to be sufficiently small. Under a oddness condition near the origin for the perturbation, a whole sequence of sign-changing solutions, which converges to zero, is produced.
\end{abstract}
\vspace{2ex}
\noindent\textbf{Keywords:} $(p,q)$-Laplacian, asymmetric perturbation, concave term, extremal constant-sign and nodal solution.
\vspace{2ex}

\noindent\textbf{AMS Subject Classification:} 35J20, 35J92, 58E05.
\section{Introduction}
Let $\Omega$ be a bounded domain in $\R^N$ with a $C^2$-boundary $\partial\Omega$, let $1<s<q\leq p<+\infty$, and let $\mu\in\R^+_0$. Consider the Dirichlet problem
\begin{equation}\label{problema}
\left\{
\begin{array}{ll}
-\Delta_p u-\mu \Delta_q u=\lambda |u|^{s-2}u+f(x, u)\quad & \mbox{in}\quad\Omega,\\
u=0\quad & \mbox{on}\quad\partial\Omega,\\
\end{array}
\right.
\end{equation}
where $\Delta_r$, $r>1$, denotes the $r$-Laplace operator, namely
$$\Delta_r u:={\rm div}(|\nabla u|^{r-2}\nabla u)\quad\forall\, u\in W^{1,r}_0(\Omega),$$
$p=q$ iff $\mu=0$, $\lambda$ is a real parameter, while $f:\Omega\times\R\to\R$ satisfies Carath\'{e}odory's conditions.

The non-homogeneous differential operator $Au:=\Delta_p u+\Delta _q u$ that appears in \eqref{problema} is usually called $(p,q)$-Laplacian. It stems from a wide range of important applications, including biophysics \cite{Fi}, plasma physics \cite{Wil}, reaction-diffusion equations \cite{Ar,CI}, as well as models of elementary particles \cite{BDFP,BFP,De}.

This paper treats the existence of multiple solutions, with a precise sign information, to \eqref{problema} when, roughly speaking,
\begin{itemize}
\item[1)] $\lambda>0$ is suitably small, and
\item[2)] $t\mapsto f(x,t)$ exhibits an asymmetric behavior as $t$ goes from $-\infty$ to $+\infty$.
\end{itemize}
We will assume that, for an appropriate constant $C>0$,
$$-C\leq\liminf_{t\to-\infty}\frac{f(x,t)}{|t|^{p-2}t}\leq\limsup_{t\to-\infty}\frac{f(x,t)}{|t|^{p-2}t}\leq\lambda_{1,p}\leq\liminf_{t\to+\infty}\frac{f(x,t)}{t^{p-1}}\leq\limsup_{t\to+\infty}\frac{f(x,t)}{t^{p-1}}\leq C$$
uniformly in $x\in\Omega$, where $\lambda_{1,p}$ indicates the first eigenvalue of $(-\Delta_p, W^{1,p}_0(\Omega))$. Hence,  $f(x,\cdot)$ grows $(p-1)$-linearly at $\pm\infty$ and only a partial interaction with $\lambda_{1,p}$ is allowed (nonuniform non-resonance).

Since $s<q\leq p$, the term $t\mapsto\lambda|t|^{s-2}t$ represents a parametric `concave' contribution inside the reaction of \eqref{problema}. 

Under 1), 2), and a further hypothesis involving the rate of $f(x,\cdot)$ near zero, Problem \eqref{problema} admits four nontrivial $C^1_0(\overline{\Omega})$-solutions, two positive, one negative, and the remaining nodal; see Theorem \ref{foursol}. If, moreover, $t\mapsto f(x,t)$ turns out to be odd in a neighborhood of zero then there exists a whole sequence $\{u_n\}$ of nodal solutions such that $u_n\to 0$ in $C^1(\overline{\Omega})$; cf. Theorem \ref{seqsol}.

The adopted approach exploits variational methods, truncation techniques, as well as results from Morse theory. Regularity is a standard matter.

Many recent papers have been devoted to elliptic problems with either 
\begin{itemize}
\item $p$-Laplacian and asymmetric nonlinearity (see, e.g., \cite{DMP,MaPa3,MaPa4,MaPa5,MoMoPa1} and the references therein), or
\item $(p,q)$-Laplacian and symmetric reaction (see for instance \cite{CMP,MMP,MaPa2} and the references given there).
\end{itemize}
On the contrary, to the best of our knowledge, few articles treat equations driven by the $(p,q)$-Laplace operator and an asymmetric nonlinearity. Actually, we can only mention \cite{PaRa}, where $\mu:=1$, $q:=2$, the parametric concave term does not appear, $f$ satisfies somewhat different assumptions, and a complete sign information on the solutions is not performed. A wider bibliography on these topics can be found in the survey paper \cite{MM}.
\section{Preliminaries}
Let $(X,\|\cdot\|)$ be a real Banach space and let $X^*$ be its topological dual, with duality bracket $\langle\cdot, \cdot\rangle$. An operator $A:X\to X^*$ is called of type $({\rm S})_+$ provided
$$x_n\rightharpoonup x\quad\text{in X,}\quad\limsup_{n\to+\infty}\langle A(x_n),x_n-x\rangle\leq 0\quad \implies\quad
x_n\to x\quad\text{in X.}$$
For $\varphi\in C^1(X)$ and $c\in \R$, put
\[
\varphi^c:=\{x\in X:\varphi(x)\leq c\},\quad K_\varphi:=\{x\in X:\varphi'(x)=0\},\quad
K^c_\varphi:=\{x\in K_\varphi:\varphi(x)=c\}.
\]
Given an isolated critical point $\bar x\in K^c_\varphi$, we define  the $k$-th critical group of $\varphi$ at $\bar x$ as
\[
C_k(\varphi, \bar x):=H_k(\varphi^c\cap U,\varphi^c \cap U\setminus\{\bar x\}),\quad k\in\mathbb{N}_0,
\]
where $U$ is any neighborhood of $\bar x$ such that $K_\varphi\cap\varphi^c\cap U=\{\bar x\}$ and $H_k(A, B)$ denotes the
$k$-th relative singular homology group for the pair $(A, B)$ with integer coefficients. The excision property of singular homology ensures that this definition does not depend on the choice of $U$; see \cite{MoMoPa2} for details.

We say that $\varphi$ satisfies the Cerami condition when
\begin{itemize}
\item[(C)]{\em Every sequence $\{x_n\}\subseteq X$ such that $\{\varphi(x_n)\}$ is bounded and $(1+\Vert x_n\Vert)\varphi'(x_n)\to 0$ in $X^*$ admits a strongly convergent subsequence.}
\end{itemize}
The following version \cite{MoMoPa2} of the mountain pass theorem will be employed.
\begin{theorem}\label{mpth}
If $\varphi\in C^1(X)$ satisfies $(C)$, $x_0, x_1\in X$, $0<\rho<\|x_0-x_1\|$,
\[
\max\{\varphi(x_0), \varphi(x_1)\}<m_\rho:=\inf_{\|x-x_0\|=\rho}\varphi(x),
\]
and
\[
c:=\inf_{\gamma\in\Gamma}\max_{t\in[0,1]}\varphi(\gamma(t)),\quad\text{where}\quad\Gamma:=\{\gamma\in C^0([0, 1], X):\;
\gamma(0)=x_0, \gamma(1)=x_1\},
\]
then: $m_\rho\le c$; $K^c_\varphi$ is nonempty; $C_1(\varphi, \bar x)\neq 0$ provided $\bar x\in K_\varphi^c$ turns out to be isolated.
\end{theorem}
Hereafter, $\Omega$ will denote a fixed bounded domain in $\R^N$ with a $C^2$-boundary $\partial\Omega$. Let $u,v:\Omega\to\R$ be measurable and let $t\in\R$. The symbol $u\leq v$ means $u(x)\leq v(x)$ for almost every $x\in\Omega$, $t^\pm:=\max\{\pm t,0\}$, $u^\pm(\s):= u(\s)^\pm$. If $p\in [1,+\infty)$ then $p':=p/(p-1)$ is the conjugate exponent of $p$ and $p^*$ indicates the Sobolev conjugate in dimension $N$, namely
\[
p^*=
\begin{cases}
\frac{Np}{N-p}&\text{when $p<N$},\\
\text{any $q>1$}&\text{for $p=N$},\\
+\infty&\text{otherwise}.
\end{cases}
\]
Set, provided $r\in [1, +\infty]$,
\[
L^r_+(\Omega)=\{u\in L^r(\Omega): u\ge 0 \text{ a.e. in $\Omega$}\}.
\]
If $r<+\infty$ then, as usual,
\[
\|u\|_r:=\left(\int_{\Omega} |\nabla u|^r\, dx\right)^{1/r}, \quad u\in W^{1,r}_0(\Omega),\quad\mbox{and}\quad
|u|_r:= \left(\int_{\Omega} |u|^r\, dx\right)^{1/r}, \quad u\in L^{r}(\Omega).
\]
 $W^{-1,r'}(\Omega)$ denotes the dual space of $W^{1,r}_0(\Omega)$ while $A_r:W^{1,r}_0(\Omega)\to W^{-1,r'}(\Omega)$ is the nonlinear operator stemming from the negative $r$-Laplacian, i.e.,
\[
\langle A_r(u),v\rangle:=\int_\Omega|\nabla u|^{r-2}\nabla u\cdot\nabla v\, dx\quad\forall\, u,v\in W^{1,r}_0(\Omega)\, .
\]
It is known \cite[Section 6.2]{GaPa} that $A_r$ turns out to be bounded, continuous, strictly monotone, as well as of type $({\rm S})_+$.

Given $\xi\in L^\infty_+(\Omega)\setminus \{0\}$, we define
\begin{equation}\label{lambda1}
\lambda_{1, r}(\xi):=\inf\left\{\frac{\int_\Omega |\nabla u|^r\, dx}{\int_\Omega \xi \, |u|^r\, dx}:\, u\in W^{1,r}_0(\Omega), u\neq 0\right\}.
\end{equation}
When no confusion can arise, simply write $\lambda_{1, r}:=\lambda_{1, r}(1)$. Some basic properties of $\lambda_{1,r}(\xi)$ and its eigenfunctions are listed below.
\begin{proposition} \label{autovalori}
Let $1<r<+\infty$ and let $\xi\in L^\infty_+(\Omega)\setminus\{0\}$. Then:
\begin{enumerate}
\item $\lambda_{1,r}(\xi)$ is positive and attained on a positive function $\hat u_{1,r}\in W^{1,r}_0(\Omega)$, which fulfills $|\hat u_{1,r}|_r=1$  as well as
\begin{equation}\label{eig}
 A_r(u)=\lambda_{1, r}(\xi)\, \xi\, |u|^{r-2}u.
\end{equation}
\item Solutions to \eqref{eig} coincide with  minima of \eqref{lambda1} and form a one-dimensional linear space.
%
%
\item The function $\xi\mapsto \lambda_{1, r}(\xi)$ is monotone (strictly) decreasing with respect to the a.e. ordering of $L^\infty_+(\Omega)$. 
\end{enumerate}
\end{proposition}
Through the compactness of the embedding $W^{1,r}_0(\Omega)\hookrightarrow L^r(\Omega)$ one can verify \cite[p. 356]{PK} the next result.
\begin{proposition}\label{bb}
If $\xi\in L^\infty_+(\Omega)\setminus\{\lambda_{1, r}\}$ and $\xi\le \lambda_{1, r}$ then there exists a constant $c(\xi)>0$ such that 
\[
\|u\|_r^r-\int_\Omega\xi \, |u|^r\, dx\ge c(\xi)\|u\|_r^r\quad \forall\, u\in W^{1, r}_0(\Omega).
\]
\end{proposition}
We will also employ the linear space
\[
C^1_0(\overline{\Omega}):=\{u\in C^1(\overline{\Omega}):\; u\restr_{\partial\Omega}=0\},
\]
which is complete with respect to the standard $C^1(\overline{\Omega})$-norm. Its positive cone
\[
C_+:=\{u\in C^1_0(\overline{\Omega}):\; u(x)\geq 0\text{ in $\overline{\Omega}$}\}
\]
has a nonempty interior given by
\[
{\rm int } (C_+)=\left\{u\in C_+:\; u(x)>0\;\;\forall\, x\in\Omega,\;\frac{\partial u}{\partial n}(x)<0\;\;\forall\, x\in\partial\Omega
\right\}.
\]
Here, $n(x)$ denotes the outward unit normal to $\partial\Omega$ at $x$.

Suppose $g:\Omega\times\R\to \R$ is a Carath\'{e}odory function growing sub-critically, i.e.,
$$|g(x,t)|\leq c(1+|t|^{r-1})\quad \text{in}\quad\Omega\times\R,$$
where $c>0$, $1\leq r<p^*$. Write, as usual, $G(x,t):=\int_0^t g(x,\tau)\, d\tau$ and consider the $C^1$-functional $\varphi: W^{1,p}_0 (\Omega)\to\R$ defined by
\[
\varphi(u):=\frac{1}{p}\Vert u\Vert_p^p+\frac{\mu}{q}\Vert u\Vert_q^q-\int_\Omega G(x, u(x))\, dx\quad \forall\, u\in
W^{1,p}_0(\Omega),
\]
with $1<q\leq p$ and $\mu\ge 0$. The next result \cite{GaPa2} establishes a relation between local minimizers of $\varphi$ in  $C^1_0(\overline{\Omega})$ and in $W^{1,p}_0(\Omega)$.
\begin{proposition}\label{locmin}
If $u_0\in W^{1,p}_0(\Omega)$ is a local $C^1_0(\overline{\Omega})$-minimizer of $\varphi$, then $u_0\in C^{1,\alpha}_0 (\overline{\Omega})$ for some $\alpha\in (0,1)$ and $u_0$ turns out to be a local $W^{1,p}_0(\Omega)$-minimizer of
$\varphi$.
\end{proposition}
\section{Solutions of constant sign}
In this section we will construct three nontrivial constant-sign solutions to Problem \eqref{problema} provided the parameter is small enough. From now on, everywhere in $\Omega$ stands for almost everywhere and $q=p$ iff $\mu=0$.

The hypotheses on the reaction $f$ will be as follows. 
\begin{itemize}
\item[(${\rm h}_0$)] $f:\Omega\times \R\to \R$ is a Carath\'{e}odory function such that
\[
|f(x, t)|\le C(1+|t|^{p-1})\quad\forall\, (x, t)\in \Omega\times \R,
\]
where $C\in\R^+$.
\item[(${\rm h}_1$)] There exists $\xi_1\in L^\infty_+(\Omega)\setminus\{\lambda_{1,p}\}$ satisfying
\[
\lambda_{1,p}\leq\xi_1(x)\le \liminf_{t\to +\infty} \frac{f(x, t)}{t^{p-1}}\quad\mbox{uniformly in $x\in \Omega$.}
\]
\item[(${\rm h}_2$)] There is $\xi_2\in L^\infty_+(\Omega)\setminus\{\lambda_{1,p}\}$ such that
\[
\limsup_{t\to -\infty} \frac{f(x, t)}{|t|^{p-2}\,t}\leq\xi_2(x)\leq\lambda_{1,p}\quad\mbox{uniformly with respect to $x\in\Omega$.}
\]
\item[(${\rm h}_3$)] There exist $\delta_0,\theta_0\in (0, 1)$ fulfilling
\[
0\le f(x, t)\,t\le\mu\lambda_{1,q}\,|t|^q+\theta_0\lambda_{1,p}|t|^p\quad\forall\, (x, t)\in \Omega\times [-\delta_0,\delta_0].
\]
\end{itemize}
\begin{remark}
It should be noted that $({\rm h}_0)$--$({\rm h}_2)$ entail
$$-C\leq \liminf_{t\to -\infty}\frac{f(x, t)}{|t|^{p-2}\,t}\leq\xi_2(x)\leq\lambda_{1,p}\leq\xi_1(x)\leq\limsup_{t\to +\infty}\frac{f(x, t)}{t^{p-1}}\le C.$$
\end{remark}
If $1<s<q$ and $\lambda\in\R^+$, we put
\[
g_\lambda(x, t):= \lambda |t|^{s-1}t+f(x, t),
\]
which still satisfies a growth condition like $({\rm h}_0)$, but with a different positive constant depending on $\lambda$, say $C_\lambda$, and 
\[
G_\lambda(x, t):=\int_0^tg_\lambda(x, \tau)\, d\tau.
\]
The energy functional $\varphi_\lambda\in C^{1}(W^{1, p}_0(\Omega))$ that stems from Problem \eqref{problema} is defined by
\[
\varphi_\lambda(u):=\frac{1}{p}\Vert u\Vert_p^p+\frac{\mu}{q}\Vert u\Vert_q^q-
\int_\Omega G_\lambda(x, u(x))\, dx\quad \forall\, u\in W^{1,p}_0(\Omega).
\]
Suitable truncations of it will be employed. With this aim, set
\[
g_{\lambda}^+(x, t):=g_\lambda(x,t^+), \quad g_{\lambda}^-(x,t):=g_\lambda(x,-t^-),\quad G_{\lambda}^\pm(x, t):=
\int_0^tg_{\lambda}^{\pm}(x,\tau)\, d\tau.
\]
Evidently, $G_{\lambda}^{+}(x, t)=G_\lambda(x,t^+)$, $G_{\lambda}^{-}(x,t)=G_\lambda(x,-t^-)$, and the associated
functionals
\[
\varphi_\lambda^\pm(u):=\frac{1}{p}\Vert u\Vert_p^p+\frac{\mu}{q}\Vert u\Vert_q^q
-\int_\Omega G_{\lambda}^\pm(x, u(x))\, dx,\quad  u\in W^{1,p}_0(\Omega),
\]
turn out to be  $C^1$ as well. Likewise the proof of \cite[Theorem 4.1]{MaPa2}, using the nonlinear regularity theory developed in \cite{LU,L}, the strong maximun principle, and the Hopf boundary point lemma \cite[pp. 111 and 120]{PS}, yields
\begin{proposition}\label{nlrt}
Under $({\rm h}_0)$ and $({\rm h}_3)$, nontrivial critical points for $\varphi_\lambda^+$ (resp., $\varphi_\lambda^-$) actually are critical points of $\varphi_\lambda$ and belong to ${\rm int} (C_+)$ (resp., $-{\rm int} (C_+)$) . 
\end{proposition}
\begin{lemma}\label{cphipm}
If $({\rm h}_0)$--$({\rm h}_2)$ hold true then 
\begin{enumerate}
\item $\varphi_\lambda^+$ satisfies Condition $({\rm C})$.
\item $\varphi_\lambda^-$ is coercive (hence it fulfills the Cerami condition too).
\end{enumerate}
\end{lemma}
\begin{proof}
{\em 1.}  Let $\{u_n\}\subseteq W^{1,p}_0(\Omega)$ be such that $\{\varphi_\lambda^+(u_n)\}$ is bounded and 
\[
(1+\|u_n\|_p)(\varphi_\lambda^+)'(u_n)\to 0\quad \text{in}\quad W^{-1, p'}(\Omega).
\]
Since the embedding $W^{1,p}_0(\Omega)\hookrightarrow L^p(\Omega)$ is compact, while $A_p+\mu A_q$ enjoys property $(S)_+$, it suffices to show that $\{u_n\}$ is bounded. One has
\begin{equation}\label{Ccond}
\left|\langle A_p(u_n), v\rangle +\mu\langle A_q(u_n), v\rangle-\int_\Omega g_\lambda^+(x, u_n)\, v\, dx\right|\le \eps_n\frac{\|v\|_p}{1+\|u_n\|_p}\quad\forall\, v\in W^{1,p}_0(\Omega),
\end{equation}
where $\eps_n\to 0^+$. Letting $v:=-u_n^-$ yields
\[
\|u_n^-\|_p+\mu\|u_n^-\|_q\le \eps_n,
\]
so that $u_n^-\to 0$ and 
\begin{equation}\label{kl}
\left|\langle A_p(u_n^+), v\rangle +\mu\langle A_q(u_n^+), v\rangle-\int_\Omega g_\lambda^+(x, u_n^+)\, v\, dx\right|\le \eps_n'\|v\|_p
\end{equation}
for some $\eps_n'\to 0^+$. Suppose $\|u_n^+\|_p\to +\infty$ and put $w_n:=u_n^+/\|u_n^+\|_p$. From $\|w_n\|_p\equiv 1$ it follows, up to subsequences,
\begin{equation}\label{wn}
w_n\weakto w\quad\text{in}\quad W^{1,p}_0(\Omega), \qquad w_n\to w\quad\text{in}\quad L^p(\Omega),\qquad w\geq 0.
\end{equation}
Moreover, 
\[
|g_\lambda^+(\cdot , u_n^+)|_{p'}^{p'}\le C_\lambda\int_\Omega (1+|u_n^+|^p)\, dx\le C_\lambda\big(|\Omega|+\lambda_{1, p}^{-1}\|u_n^+\|_p^p\big)
\]
because $f$ satisfies (${\rm h}_0$). This implies 
\[
\left|\frac{g_\lambda^+(\cdot,u_n^+)}{\|u_n^+\|_p^{p-1}}\right|_{p'}\le C_\lambda^{1/p'} \left(\frac{|\Omega|}{\|u_n^+\|_p^{p}}+\frac{1}{\lambda_{1, p}}\right)^{1/p'}.
\]
Since the right-hand side is bounded, we may suppose 
\[
\frac{g_\lambda^+(\cdot , u_n^+)}{\|u_n^+\|_p^{p-1}}\weakto h\quad \text{in}\quad L^{p'}(\Omega).
\]
Recalling that $s<p$, namely $({\rm h}_1)$ holds true for $g_\lambda^+$, and proceeding as in \cite[pp. 317-318]{MoMoPa2} produces
\begin{equation}\label{etaxione}
h=\eta\, w^{p-1}\quad \text{for some $\eta\in L^\infty(\Omega)$ with $\xi_1(x)\leq\eta(x)\leq C$}.
\end{equation}
Through \eqref{kl} we then have
\begin{equation}\label{Apwn}
\left|\langle A_p(w_n), v\rangle +\mu\|u_n^+\|_p^{q-p}\langle A_q(w_n), v\rangle-\int_\Omega \frac{g_\lambda^+(x, u_n^+)}{\|u_n^+\|_p^{p-1}}\, v\, dx\right|\le \eps_n'\frac{\|v\|_p}{\|u_n^+\|^{p-1}}.
\end{equation}
Now, choose $v:=w_n-w$ and use \eqref{wn} to arrive at
\[
\lim_{n\to+\infty}\langle A_p(w_n), w_n-w\rangle=0.
\]
Therefore, by \cite[Proposition 2.72]{MoMoPa2}, $w_n\to w$ in $W^{1,p}_0(\Omega)$, whence $\|w\|_p=1$. Via \eqref{Apwn} we thus obtain, letting $n\to+\infty$,
\[
\langle A_p(w), v\rangle=\int_{\Omega}\eta \, w^{p-1}v\, dx\quad \forall\, v\in W^{1,p}_0(\Omega),
\]
i.e., $\lambda=1$ is an eigenvalue for the problem
\[
-\Delta_p u=\lambda\, \eta\, |u|^{p-2}u\quad\mbox{in}\quad\Omega,\quad u=0\quad\mbox{on}\quad\partial\Omega
\] 
associated with the eigenfunction $w$. However, due to Item {\em 4)} of Proposition \ref{autovalori}, $({\rm h}_1)$, and \eqref{etaxione},
\[
1=\lambda_{1,p}(\lambda_{1, p})>\lambda_{1, p}(\xi_1)\ge \lambda_{1, p}(\eta).
\]
Point {\em 3)} in the same result ensures that $w$ changes sign, contradicting $w\geq 0$.

{\em 2.} By (${\rm h}_0$) and (${\rm h}_2$), for every $\eps>0$ there exists a constant $C_\eps>0$ such that
\[
F(x, t)\le \frac{1}{p}(\xi_2(x)+\eps)|t|^p+C_\eps\quad \forall\, (x, t)\in \Omega\times (-\infty, 0].
\]
Thus, on account of Proposition \ref{bb},
\[
\begin{split}
\varphi^-_\lambda(u)&\ge \frac{1}{p}\|u\|_p^p+\frac{\mu}{q}\|u\|_q^q-\frac{\lambda}{s}|u|_s^s-\int\frac{1}{p}(\xi_2+\eps)|u^-|^p dx-C_\eps|\Omega|\\
&\ge \frac{1}{p}\|u^+\|_p^p+\frac{1}{p}\left(\|u^-\|_p^p-\int_\Omega \xi_2\, |u^-|^p\, dx\right) -\frac{\eps}{p}|u^-|_p^p-\frac{\lambda}{s}|u|_s^s-C_\eps|\Omega|\\
&\ge  \frac{1}{p}\|u^+\|_p^p+\frac{1}{p}\left(c(\xi_2)-\frac{\eps}{\lambda_{1, p}}\right)\|u^-\|_p^p-\frac{\lambda}{s}|u|_s^s-C_\eps|\Omega|.
\end{split}
\]
Choosing $\eps:=\lambda_{1,p}c(\xi_2)/2$ and recalling that $s<p$ finally provides the desired coercivity property.
\end{proof}

With slight modifications one can verify the next lemma.

\begin{lemma}
Under $({\rm h}_0)$--$({\rm h}_2)$, the functional $\varphi_\lambda$ satisfies Condition $({\rm C})$ for all $\lambda>0$.
\end{lemma}
\begin{proof}
Fix $\lambda>0$. Let $\{u_n\}\subseteq W^{1,p}_0(\Omega)$ be such that $\{\varphi_\lambda(u_n)\}$ is bounded and
\[
(1+\|u_n\|_p)(\varphi_\lambda)'(u_n)\to 0\quad \text{in}\quad W^{-1, p'}(\Omega).
\]
Then \eqref{Ccond} holds with $g_\lambda$ instead of $g_\lambda^+$. Choosing $v:=-u_n^-$, it furnishes
\begin{equation}\label{cepsn}
\|u^-_n\|_p^p+\mu\|u_n^-\|_q^q-\lambda |u_n^-|_s^s-\int_\Omega f(x, -u_n^-)(-u_n^-)\, dx\le \eps_n\, ,\quad n\in\mathbb{N},
\end{equation}
where $\eps_n\to 0^+$. Thanks to $({\rm h}_0)$ and $({\rm h}_2)$, for every $\eps>0$ there exists a constant $C_\eps>0$ such that
\[
f(x, t) t\le (\xi_2(x)+\eps)|t|^p+C_\eps\quad\forall\, (x,t)\in\Omega\times]-\infty,0].
\]
So, the proof of Conclusion {\em 2} in the previous lemma carries over, giving the coerciveness of the functional
\[
u\mapsto \|u^-\|_p^p+\mu\|u^-\|_q^q-\lambda |u^-|_s^s-\int_\Omega f(x, -u^-)(-u^-)\, dx,\quad u\in W^{1,p}_0(\Omega). 
\]
Hence, due to \eqref{cepsn}, the sequence $\{u_n^-\}$ has to be bounded. To check that the same holds for $\{u_n^+\}$, suppose on the contrary $\|u_n^+\|_p\to +\infty$ and put  $w_n:=u_n/\|u_n^+\|_p$. Obviously, $\{w_n\}$ turns out to be bounded, because so is $\{ u_n^-\}$. Moreover, $w_n^-\to 0$ while, along a subsequence when necessary,
\[
w_n\rightharpoonup w\quad \text{in}\quad W^{1,p}_0(\Omega),\quad w_n\to w\quad\mbox{in}\quad L^p(\Omega).
\]
As before, via $({\rm h}_0)$ we see that $\left\{g_\lambda(\cdot , u_n)\|u_n^+\|_p^{1-p}\right\}$ is bounded in $L^{p'}(\Omega)$.
Now, divide the present version of \eqref{Ccond}  by $\|u_n^+\|^{p-1}$,  test with $v:=w_n-w$, use the inequality $q<p$, and let $n\to +\infty$ to achieve
\[
\lim_{n\to+\infty}\langle A_p(w_n), w_n-w\rangle=0,
\]
which implies $w_n\to w$ in $W^{1, p}_0(\Omega)$. Consequently, $w\ge 0$ and $\|w\|_p\ge 1$. Since
\[
\left|\frac{g_\lambda (\cdot , -u^-)}{\|u^+_n\|_p^{p-1}}\right|_{p'}\le C_\lambda^{1/p'}\left(\frac{|\Omega|}{\|u^+_n\|_p^p}+\frac{1}{\lambda_{1,p}}\frac{\|u_n^-\|_p^p}{\|u_n^+\|_p^p}\right)^{1/p'}\to 0,
\]
we have 
\begin{equation}\label{difference}
\frac{g_\lambda(\cdot , u_n)}{\|u_n^+\|_p^{p-1}}-\frac{g_\lambda(\cdot , u_n^+)}{\|u_n^+\|_p^{p-1}}\to 0\quad \text{in}\quad L^{p'}(\Omega).
\end{equation}
The same arguments of \cite[pp. 317-318]{MoMoPa2} yield here
\[
\frac{g_\lambda(\cdot , u_n^+)}{\|u_n^+\|_p^{p-1}}\weakto \eta\, w^{p-1}\quad \text{in}\quad L^{p'}(\Omega),
\]
with appropriate $\eta\in L^\infty(\Omega)$ fulfilling $\xi_1(x)\le \eta(x)\le C$. Thanks to \eqref{difference}, this holds true also for 
$\left\{g_\lambda(\cdot , u_n)\|u_n^+\|_p^{1-p}\right\}$. Hence, from
\[
\left|\langle A_p(w_n), v\rangle +\mu\|u_n^+\|_p^{q-p}\langle A_q(w_n), v\rangle-\int_\Omega \frac{g_\lambda(x, u_n)}{\|u_n^+\|_p^{p-1}}\, v\, dx\right|\le \eps_n'\frac{\|v\|_p}{\|u_n^+\|^{p-1}}
\]
(cf. \eqref{Apwn}) it follows, when $n\to+\infty$,
\[
\langle A_p(w), v\rangle=\int_\Omega \eta \, w^{p-1}v\, dx\quad \forall\, v\in W^{1,p}_0(\Omega).
\]
Now the proof goes on exactly as the one of Item {\em 1} in Lemma \ref{cphipm}.
\end{proof}
\begin{lemma}\label{pos}
If $({\rm h}_0)$ and $({\rm h}_3)$ are satisfied then there exists a constant $\lambda^*>0$ such that to every $\lambda\in (0, \lambda^*)$ corresponds a $\rho_\lambda\in\R^+$ complying with
\[
m_\lambda:=\inf_{\|u\|_p=\rho_\lambda}\varphi_\lambda^+(u)>0.
\]
\end{lemma}
\begin{proof}
Fix any $r\in (p, p^*)$. Through $({\rm h}_0)$ and $({\rm h}_3)$ we obtain
\[
f(x, t)\le\mu\lambda_{1, q} t^{q-1}+\theta_0\lambda_{1,p}t^{p-1}+C_r t^{r-1},\quad (x, t)\in \Omega\times [0, +\infty),
\]
which, when integrated, entails
\[
F(x, t)\le\mu\frac{\lambda_{1, q}}{q}t^q+\theta_0\frac{\lambda_{1,p}}{p}t^p+\frac{C_r}{r}t^r\quad\mbox{in}\quad\Omega\times [0, +\infty).
\]
Here, $C_r\in\R^+$. By the Sobolev, H\"older, and Poincar\'e inequalities one has
\[
\begin{split}
\varphi_\lambda^+(u)&\ge \frac{1}{p}\|u\|_p^p+\frac{\mu}{q}\|u\|_q^q-\frac{\lambda}{s}|u^+|_s^s
-\mu\frac{\lambda_{1, q}}{q}|u|_q^q-\theta_0\frac{\lambda_{1,p}}{p}|u|^p_p-\frac{C_r}{r}|u|_r^r\\
&\ge \frac{1-\theta_0}{p}\|u\|_p^p-\frac{\lambda}{s}|u^+|_s^s-\frac{C_r}{r}|u|_r^r\\
&\ge \frac{1-\theta_0}{p}\|u\|_p^p-C_1(\Omega)\frac{\lambda}{s}|u^+|_{p^*}^s-\frac{C_2(\Omega)C_r}{r}|u|_{p^*}^r\\
&\ge \left[\frac{1-\theta_0}{p}-\bar C_1\lambda\|u\|_p^{s-p}-\bar C_2\|u\|_p^{r-p}\right]\|u\|_p^p
\end{split}
\]
for appropriate positive constants $\bar C_1$, $\bar C_2$. Letting $\|u\|_p=\lambda^{\frac{1}{r-s}}$ yields
\[
\varphi_\lambda^+(u)\ge \left[\frac{1-\theta_0}{p}-\bar C_1\lambda^{1-\frac{p-s}{r-s}}-\bar C_2\lambda^{\frac{r-p}{r-s}}\right]\lambda^{\frac{1}{r-s}}=\left[\frac{1-\theta_0}{p}-(\bar C_1+\bar C_2)\lambda^{\frac{r-p}{r-s}}\right]\lambda^{\frac{1}{r-s}}.
\]
This immediately brings the conclusion, because $s<p<r$ and $0<\theta_0<1$.
\end{proof}

From now on, $\lambda^*$ will denote the real number just found. 
 
\begin{lemma}\label{findtau}
Suppose $({\rm h}_0)$--$({\rm h}_1)$ hold true. Then
\[
\lim_{\tau\to +\infty} \varphi^+_\lambda(\tau\hat u_{1,p})=-\infty,
\]
with $\hat u_{1,p}$ as in Proposition \ref{autovalori}.
\end{lemma}
\begin{proof}
Thanks to $({\rm h}_0)$--$({\rm h}_1)$, for every $\eps>0$ there exists a constant $C_\eps>0$ such that
\begin{equation}\label{useone}
F(x, t)\ge \frac{\xi_1(x)-\eps}{p}t^p-C_\eps\quad\forall\,(x, t)\in \Omega\times [0, +\infty).
\end{equation}
The properties of $\hat u_{1,p}$ and $\xi_1$ produce
\[
\int_\Omega (\xi_1-\lambda_{1,p})\hat u_{1,p}^p\, dx>0.
\]
Choose $\eps>0$ fulfilling
\[
\theta:=\int_\Omega (\xi_1-\lambda_{1, p})\hat u_{1,p}^p\, dx-\eps \int_\Omega \hat u_{1,p}^p\, dx>0.
\]
Since $\|\hat u_{1,p}\|_p^p=\lambda_{1, p}|\hat u_{1,p}|_p^p=\lambda_{1, p}$, via \eqref{useone} we get
\[
\begin{split}
\varphi_\lambda^+(\tau\hat u_{1,p})&\le \frac{\tau^p}{p}\|\hat u_{1,p}\|_p^p+\mu\frac{\tau^q}{q}\|\hat u_{1,p}\|_q^q-\frac{\tau^p}{p}\int_\Omega(\xi_1-\eps)\hat u_{1,p}^p\, dx-\lambda\frac{\tau^s}{s}|\hat u_{1,p}|_s^s+C_\eps|\Omega|\\
&\le-\theta\frac{\tau^p}{p}+\mu\frac{\tau^q}{q}\|\hat u_{1,p}\|_q^q-\lambda\frac{\tau^s}{s}|\hat u_{1,p}|_s^s
+C_\eps|\Omega|
\end{split}
\]33
for all $\tau>0$. The conclusion follows from $q<p$.
\end{proof}

Now, critical point arguments will provide three constant-sign solutions.

\begin{theorem}\label{fixedsign}
Let $({\rm h}_0)$--$({\rm h}_3)$ be satisfied. Then:
\begin{enumerate}
\item For every $\lambda\in (0, \lambda^*)$, Problem \eqref{problema} admits two positive solutions $u_0,u_1\in
{\rm int}(C_+)$.
\item For every $\lambda>0$ there exists a negative solution $u_2\in-{\rm int}(C_+)$ to \eqref{problema}.
\end{enumerate}
\end{theorem}
\begin{proof}
{\em 1}. Pick $\lambda\in (0, \lambda^*)$. Lemma \ref{findtau} gives a $\tau\in\R^+$ so large that $\varphi_\lambda^+(\tau \hat u_{1,p})<0$. On account of Lemmas \ref{cphipm} and \ref{pos}, Theorem \ref{mpth} applies to $\varphi_\lambda^+$. Thus, there is $u_0\in W^{1,p}_0(\Omega)$ fulfilling
$$(\varphi_\lambda^+)'(u_0)=0,\quad\varphi_\lambda^+(u_0)\geq m_\lambda>0,$$
whence $u_0\neq 0$. By Proposition \ref{nlrt}, the function $u_0$ turns out to be a solution of \eqref{problema} lying in ${\rm int}(C_+)$ Next, define
$$B_{\rho_\lambda}:=\{u\in W^{1,p}_0(\Omega): \Vert u\Vert_p<\rho_\lambda\},$$
where $\rho_\lambda$ comes from Lemma \ref{pos}. A standard procedure based on the weak sequential lower semicontinuity of $\varphi_\lambda^+$ ensures that this functional attains its minimum at some $u_1\in\overline{B}_{\rho_\lambda}$. Fix $w\in{\rm int}(C_+)$ and choose $\tau_1>0$ complying with
$$\Vert\tau_1 w\Vert_p<\rho_\lambda\, ,\quad \tau_1\sup_{x\in\overline{\Omega}} w(x)\leq\delta_0.$$
Thanks to $({\rm h}_3)$ we have
$$f(x,\tau w(x))\geq 0\quad\forall\, \tau\in (0,\tau_1),$$
which easily entails
$$\varphi_\lambda^+(\tau w)\leq\frac{\tau^p}{p}\Vert w\Vert_p^p+\mu\frac{\tau^q}{q}\Vert w\Vert_q^q-\lambda\frac{\tau^s}{s}\vert w\vert_s^s<0$$
provided $\tau$ is sufficiently small (recall that $s<q<p$). Hence, a fortiori,
$$\varphi_\lambda^+(u_1)<0.$$
The above inequality brings both $u_1\neq u_0$ and $u_1\in B_{\rho_\lambda}\setminus\{0\}$. On account of \cite[Lemma 4.3]{MaMo} we thus arrive at $(\varphi_\lambda^+)'(u_1)=0$. Finally, due to Proposition \ref{nlrt}, the function $u_1$ lies in ${\rm int}(C_+)$ and solves \eqref{problema}. \\ 
{\em 2.}
$\varphi_\lambda^-$ is coercive (cf. Lemma \ref{cphipm}) and weakly sequentially lower semicontinuous. So, it attains its minimum at some $u_2\in W^{1,p}_0(\Omega)$. As before, we see that $\varphi_\lambda^-(u_2)<0$, whence $u_2\neq 0$. Since
$(\varphi_\lambda^-)'(u_2)=0$, Proposition \ref{nlrt} applies to get the conclusion.
\end{proof}
\section{Nodal solutions}
Let us first show that \eqref{problema} admits extremal constant-sign, namely a smallest positive and a biggest negative, solutions. Indeed, $({\rm h}_0)$ and $({\rm h}_3)$ yield a real number $c_0>0$ fulfilling
\[
f(x, t)t\geq -c_0|t|^p \quad \forall\, (x,t)\in\Omega\times\R.
\]
The same arguments exploited to prove \cite[Lemma 2.2]{MaPa2} ensure here that, given $\lambda>0$, the auxiliary problem
\[
-\Delta_p u-\mu\Delta_q u=\lambda |u|^{s-2}u-c_0|u|^{p-2}u\quad\mbox{in}\quad\Omega,\quad u=0\quad\mbox{on}\quad\partial\Omega
\]
has only one positive solution $\barra{u}_\lambda\in {\rm int}(C_+)$, while, by oddness, $\barra{v}_\lambda:=-\barra{u}_\lambda$ turns out to be its unique negative solution. Reasoning as made for \cite [Lemma 3.3]{MMP} we next achieve 
\begin{lemma}\label{barrier}
Under $({\rm h}_0)$--$({\rm h}_3)$, any positive (resp., negative) solution $u$ of \eqref{problema} satisfies the inequality $ \bar u_\lambda\leq u$ (resp., $u\leq\bar v_\lambda$).
\end{lemma}
These facts give rise to the following result; cf. the proof of \cite[Lemma 4.2]{MaPa2}.
\begin{lemma}\label{extsol}
Assume $({\rm h}_0)$--$({\rm h}_3)$. Then, for every $\lambda\in (0, \lambda^*)$, Problem \eqref{problema} possesses a smallest positive solution $u^+_\lambda\in{\rm int}(C_+)$ and a greatest negative solution $u^-_\lambda\in-{\rm int}(C_+)$.
\end{lemma}
We are in a position now to produce a nodal solution through  a mountain pass procedure. Set 
\[
[u^-_\lambda, u^+_\lambda]:=\{ u\in W^{1,p}_0(\Omega): u^-_\lambda\le u\le u^+_\lambda \ \text{a.e. in $\Omega$}\}.
\]
\begin{theorem}\label{nodal}
If $({\rm h}_0)$--$({\rm h}_3)$ hold true and $\lambda\in (0, \lambda^*)$ then there exists a sign-changing solution $u_3\in C^1_0 (\overline{\Omega}) \cap [u^-_\lambda, u^+_\lambda]$ to \eqref{problema}.
\end{theorem}
\begin{proof}
The proof is similar to that of \cite[Theorem 3.8]{MMP}; so, we only sketch it. Define, for every $(x,t)\in\Omega\times\mathbb{R}$, 
\begin{equation*}
\hat g_\lambda(x,t):=\left\{
\begin{array}{lll}
\lambda |u^-_\lambda(x)|^{s-2}u^-_\lambda(x)+f(x,u^-_\lambda(x)) & \mbox{ if } t<u^-_\lambda(x),\\
\lambda |t|^{s-2}t+f(x,t) & \mbox{ if } u^-_\lambda(x)\leq t\leq u^+_\lambda(x),\\
\lambda |u_\lambda^+(x)|^{s-2}u_\lambda^+(x)+ f(x,u^+_\lambda(x)) & \mbox{ if } t>u^+_\lambda(x),
\end{array}
\right.
\end{equation*}
as well as $\hat G_\lambda(x,t):=\int_0^t \hat g_\lambda(x,\tau)\, d\tau$. The associated energy functional
\[
\hat\varphi_\lambda(u):=\frac{1}{p}\Vert u\Vert^p_p+\frac{\mu}{q}\Vert u\Vert^q_q
-\int_\Omega \hat G_\lambda(x,u(x))\, dx\, ,\quad u\in W^{1,p}_0(\Omega),
\]
clearly fulfills
\beq
\label{claim1}
K_{\hat\varphi_\lambda}\subseteq [u^-_\lambda, u^+_\lambda],
\eeq
while using Proposition \ref{locmin} one verifies that both $u^+_\lambda$ and $u^-_\lambda$ turn out to be local $W^{1,p}_0(\Omega)$-minimizers of $\hat\varphi_\lambda$; see \cite{MaPa2,MMP} for details. We may suppose $K_{\hat\varphi_\lambda}$
finite, otherwise $\hat\varphi_\lambda$ (and so $\varphi_\lambda$, by standard nonlinear regularity theory) would have infinitely many critical points in $[u_\lambda^-, u_\lambda^+]\setminus \{u_\lambda^-, u_\lambda^+\}$, which brings the conclusion. Consequently, $u_\lambda^+$ and $u_\lambda^-$ are strict local minimizers. Without loss of generality, assume $\hat\varphi_\lambda(u^+_\lambda)\le \hat\varphi_\lambda(u^-_\lambda)$. The reasoning employed  to establish \cite[Proposition 29]{APS} produces here $\rho\in (0,1)$ such that
\beq\label{30}
\Vert u^-_\lambda-u^+_\lambda\Vert_p>\rho,\quad\max\{\hat\varphi_\lambda(u^+_\lambda),\hat\varphi_\lambda(u^-_\lambda)\}<m_\rho:=\inf_{\Vert u-u^-_\lambda\Vert_p=\rho}\hat\varphi_\lambda(u).
\eeq
Moreover, $\hat\varphi_\lambda$ satisfies Condition (C), because it evidently is coercive. Hence, Theorem \ref{mpth} applies, and there exists $u_3\in K_{\hat\varphi_\lambda}$ such that $m_\rho\leq \hat\varphi_\lambda(u_3)$. Combining \eqref{30} with \eqref{claim1} entails $u_3\in [u^-_\lambda, u^+_\lambda]\setminus\{u^-_\lambda, u^+_\lambda\}$ while, by Theorem \ref{mpth} again,
\begin{equation}\label{coneuthree}
C_1(\hat\varphi_\lambda, u_3)\neq 0.
\end{equation}
 Now, through $u^+_\lambda,-u^-_\lambda\in{\rm int}(C_+)$, $\hat\varphi_\lambda\restr_{[u^-_\lambda,u^+_\lambda]}=\varphi_\lambda\restr_{[u^-_\lambda, u^+_\lambda]}$, besides $0\in {\rm int}_{C^1_0(\overline{\Omega})}([u_\lambda^-, u_\lambda^+])$, we infer
\[
C_k(\hat\varphi_\lambda\restr_{C^1_0(\overline{\Omega})}, 0)
=C_k(\varphi_\lambda\restr_{C^1_0(\overline{\Omega})}, 0), \quad
k\in\mathbb{N}_0\, .
\]
Since $C^1_0(\overline{\Omega})$ is dense in $W^{1,p}_0(\Omega)$, one has
\[
C_k(\hat\varphi_\lambda\restr_{C^1_0(\overline{\Omega})}, 0)=C_k(\hat\varphi_\lambda, 0),\quad 
C_k(\varphi_\lambda\restr_{C^1_0(\overline{\Omega})}, 0)=C_k(\varphi_\lambda, 0).
\]
Let us verify that 
\begin{equation}\label{ckzero}
C_k(\hat\varphi_\lambda, 0)=0\quad\forall\, k\in\mathbb{N}_0,
\end{equation}
which will force $u_3\neq 0$ thanks to \eqref{coneuthree}. Pick any $\theta\in (s, q)$. Assumption $({\rm h}_3)$ directly yields
\[
g_\lambda(x, t)t=\lambda |t|^s +f(x, t)t\ge \lambda |t|^\theta\quad\mbox{in}\quad\Omega\times [-\delta_0, \delta_0].
\]
If $\delta_1<\delta_0$ fulfills
\[
|t|\le \delta_1\quad\implies\quad\left(\frac\theta s-1\right)\lambda |t|^s\ge\mu\lambda_{1, q} |t|^q+\theta_0\lambda_{1,p}|t|^p
\]
then from $({\rm h}_3)$ and the obvious inequality $F(x, t)\ge 0$ for $|t|\le \delta_1$ we deduce
\[
g_\lambda(x, t)t\le \lambda |t|^s+\mu\lambda_{1, q} |t|^q+\theta_0\lambda_{1,p}|t|^p\le \lambda \frac\theta s |t|^s\le \theta\left(\frac{\lambda}{s}|t|^s+F(x, t)\right)\le \theta G_\lambda(x, t).
\]
Consequently,
$$\lambda |t|^\theta\le g_\lambda(x, t)t\le  \theta G_\lambda(x, t),\quad  (x, t)\in \Omega\times [-\delta_1, \delta_1].$$
Recalling that $\hat\varphi_\lambda$ is coercive, \cite[Theorem 3.6]{MMP} can be used to achieve \eqref{ckzero}, as desired.
Finally, $u_3$ is nodal by extremality of $u^-_\lambda$ and $u^+_\lambda$ (cf. Lemma \ref{extsol}), while standard nonlinear regularity results give $u_3\in C^1_0(\overline{\Omega})$.
\end{proof}
 Theorems \ref{fixedsign} and \ref{nodal} together produce the following
\begin{theorem}\label{foursol}
Let $({\rm h}_0)$--$({\rm h}_3)$ be satisfied. Then there exists a real number $\lambda^*>0$ such that for every $\lambda\in (0, \lambda^*)$ Problem \eqref{problema} admits four nontrivial solutions: one smallest positive, $u^+_\lambda \in {\rm int}(C_+)$; a further positive solution $\bar u_\lambda\in {\rm int}(C_+)$; one greatest negative, $u^-_\lambda\in-{\rm int}(C_+)$; a nodal solution $u_\lambda\in C^1_0(\overline{\Omega})\cap [u^-_\lambda, u^+_\lambda]$. 
\end{theorem}
Our next goal is to find a whole sequence of sign-changing solutions that converges to zero. With this aim, the local symmetry condition below involving $f(x,\cdot)$ will be posited whatever $x\in\Omega$.
\begin{enumerate}
\item[$({\rm h}_4)$] The function $t\mapsto f(x,t)$ is odd on $[-\delta_1, \delta_1]$ for some $\delta_1\in\R^+$.
\end{enumerate}
\begin{theorem}\label{seqsol}
If $({\rm h}_0)$--$({\rm h}_4)$ hold true then to every $\lambda>0$ corresponds a sequence $\{u_n\}$ of nodal solutions to \eqref{problema} such that $u_n\to 0$ in $C^1(\overline\Omega)$.
\end{theorem}
\begin{proof}
Pick $\lambda>0$ and define, provided $(x,t)\in\Omega\times\R$,
\begin{equation}\label{hlambda}
h_\lambda(x, t):=\left\{
\begin{array}{lll}
-\lambda\delta_1^{s-1}-f(x,\delta_1) & \mbox{when $t<-\delta_1$},\\
\lambda |t|^{s-2}t+f(x,t) & \mbox{when $|t|\leq\delta_1$},\\
\lambda\delta_1^{s-1}+ f(x,\delta_1) & \mbox{when $t>\delta_1$},\\
\end{array}
\right.
\end{equation}
as well as $H_\lambda(x, t):=\int_0^t h_\lambda(x, \tau)\, d\tau$. By $({\rm h}_4)$, the associated energy functional
\[
\psi_\lambda(u):=\frac{1}{p}\|u\|_p^p+\frac{\mu}{q} \|u\|_q^q-\int_\Omega H_\lambda(x, u(x))\, dx,\quad u\in W^{1,p}_0(\Omega),
\]
turns out to be even, besides $C^1$ and coercive. 

Let $V\subseteq C^1_0(\overline\Omega)\subseteq W^{1, p}_0(\Omega)$ be any finite dimensional space. Bearing in mind that all norms on $V$ are equivalent, we have
\[
u\in V,\;\|u\|_p\le \rho\quad\implies\quad |u(x)|\leq\min\{\delta_0, \delta_1\},\;\; x\in\Omega,
\]
for some constant $\rho>0$ depending on $V$. Hence, $F(x, u(x))\ge 0$ thanks to $({\rm h}_3)$, which entails
\begin{equation*}
\psi_\lambda(u)\le \frac 1 p \|u\|_p^p+\frac{\mu}{q}\|u\|_q^q-\frac{\lambda}{s}|u|_s^s\quad \text{whenever $u\in V$, $\|u\|_p\le \rho$}. 
\end{equation*}
Since $s<q\leq p$, there exists $\rho'\in (0,\rho)$ such that $\psi_\lambda(u)<0$ provided $u\in V$, $\|u\|_p=\rho'$. So, Theorem 1 of \cite{K} furnishes a sequence
\begin{equation}\label{critpsi}
\{u_n\}\subseteq K_{\psi_\lambda}\cap\psi_\lambda^{-1}((-\infty,0))
\end{equation}
converging to zero in $W^{1,p}_0(\Omega)$. The nonlinear regularity theory ensures a uniform $C^{1, \alpha}(\overline\Omega)$-bound on $\{u_n\}$. Through Ascoli-Arzel\`{a}'s theorem we thus obtain $u_n\to 0$ in $C^1(\overline\Omega)$. Consequently, if $\bar u_\lambda$ is the barrier given by Lemma \ref{barrier} then 
\begin{equation}\label{propun}
\sup_{x\in\Omega}|u_n(x)|<\min\left\{\delta_1,\sup_{x\in\Omega}\bar u_\lambda(x)\right\}
\end{equation}
for any $n$ large enough. These $u_n$ evidently solve Problem \eqref{problema}, because of \eqref{hlambda}--\eqref{propun}. Due to Lemma \ref{barrier} and \eqref{propun} again, they must be nodal.
\end{proof}
\section*{Acknowledgement}
Work performed under the auspices of GNAMPA of INDAM.
\end{document}